\documentclass{amsart}
\usepackage{amssymb}
\newtheorem{theorem}{Theorem}[section]
\newtheorem{proposition}[theorem]{Proposition}
\newtheorem{lemma}[theorem]{Lemma}

\theoremstyle{definition}

\numberwithin{equation}{section}

\begin{document}
\title[A finite field analogue of the Appell series $F_4$]
{A finite field analogue of the Appell series $F_4$}


%    author one information
 %\author[short version for running head]{}

\author{Mohit Tripathi}
\address{Department of Mathematics, Indian Institute of Technology Guwahati, North Guwahati, Guwahati-781039, Assam, INDIA}
\curraddr{}
\email{m.tripathi@iitg.ac.in}

\author{Rupam Barman}
\address{Department of Mathematics, Indian Institute of Technology Guwahati, North Guwahati, Guwahati-781039, Assam, INDIA}
\curraddr{}
\email{rupam@iitg.ac.in}
\thanks{}

%    author two information

%    \subjclass is required.
\date{January 3, 2018}
\thanks{}
%    Abstract is required.
\begin{abstract}
We define a function $F_4^{\ast}$ as a finite field analogue of the classical Appell series $F_4$ using Gauss sums. We establish identities for $F_4^{\ast}$
analogous to those satisfied by the classical Appell series $F_4$.

\end{abstract}
\maketitle
\section{Introduction and statement of results}
For a complex number $a$, the rising factorial or the Pochhammer symbol is defined as $(a)_0=1$ and $(a)_k=a(a+1)\cdots (a+k-1), ~k\geq 1$.
If $\Gamma(x)$ denotes the gamma function, then we have
$(a)_k=\dfrac{\Gamma(a+k)}{\Gamma(a)}$.
For $a, b, c\in\mathbb{C}$, where $c$ is neither zero nor a negative integer, the hypergeometric series $F(a, b; c; x)$ is defined by
\begin{align}
F(a, b; c; x):=\sum_{k=0}^{\infty}\frac{(a)_k (b)_k}{(c)_k}\cdot\frac{x^k}{k!}.\notag
\end{align}
If we consider the product of two hypergeometric series $F(a, b; c; x)$ and $F(a', b'; c'; y)$, we obtain a double series.
Among them Appell's hypergeometric series of two variables
are the most important ones, namely \cite{bailey}
\begin{align}
F_{1}(a;b,b^{\prime};c;x,y) &= \sum_{m,n \geq 0}\frac{(a)_{m+n}(b)_{m}(b^{\prime})_{n}}{m!n!(c)_{m+n}}x^{m}y^{n}, |x|<1, |y|<1;\notag\\
F_{2}(a;b,b^{\prime};c,c^{\prime};x,y) &= \sum_{m,n \geq 0}\frac{(a)_{m+n}(b)_{m}(b^{\prime})_{n}}{m!n!(c)_{m}(c^{\prime})_{n}}x^{m}y^{n}, |x|+|y|<1;\notag\\
F_{3}(a,a^{\prime};b,b^{\prime};c;x,y) &= \sum_{m,n \geq 0}\frac{(a)_{m}(a^{\prime})_{n}(b)_{m}(b^{\prime})_{n}}{m!n!(c)_{m+n}}x^{m}y^{n}, |x|<1, |y|<1; \notag\\
F_{4}(a;b;c,c^{\prime};x,y) &= \sum_{m,n \geq 0}\frac{(a)_{m+n}(b)_{m+n}}{m!n!(c)_{m}(c^{\prime})_{n}}x^{m}y^{n}, |x|^{\frac{1}{2}}+|y|^{\frac{1}{2}}< 1.\notag
\end{align}
Let $p$ be an odd prime, and let $\mathbb{F}_q$ denote the finite field with $q$ elements, where $q=p^r, r\geq 1$.
Let $\widehat{\mathbb{F}_q^{\times}}$ be the group of all multiplicative
characters on $\mathbb{F}_q^{\times}$. We extend the domain of each $\chi\in \widehat{\mathbb{F}_q^{\times}}$ to $\mathbb{F}_q$ by setting $\chi(0)=0$
including the trivial character $\varepsilon$. We denote by $\overline{\chi}$ the character inverse of a multiplicative character $\chi$.
Recently, there has been a study for Appell series over finite fields. For example, see \cite{BST, bing-he-2, bing-he-1, li-li-mao}. 
Appell series $F_1, F_2, F_3$ and $F_4$ have integral representations. Using an integral representation of $F_1$, Li et. al. \cite{li-li-mao} defined a finite field
analogue of $F_1$ as follows. Let $A, A', B, B', C, C'$ be multiplicative characters on $\mathbb{F}_q$. 
For $x, y\in \mathbb{F}_q$, the finite field Appell series $F_1$ is
defined by
\begin{align}\label{bing-he-def-1}
 F_1(A; B, B'; C; x, y)=\varepsilon(xy)AC(-1)\sum_{u\in \mathbb{F}_q}A(u)\overline{A}C(1-u)\overline{B}(1-ux)\overline{B'}(1-uy).
\end{align}
Using an integral representation of $F_2$, He et. al. \cite{bing-he-1} defined the following function as a finite field
analogue of $F_2$.
\begin{align}\label{bing-he-def-2}
& F_2(A; B, B'; C, C'; x, y)\notag\\
& =\varepsilon(xy)BB'CC'(-1)\sum_{u, v\in \mathbb{F}_q}B(u)B'(v)\overline{B}C(1-u)\overline{B'}C'(1-v)\overline{A}(1-ux-vy).
\end{align}
Similarly, using an integral representation of $F_3$, He \cite{bing-he-2} defined the following function as a finite field analogue of $F_3$.
\begin{align}\label{bing-he-def-3}
 &F_3(A, A'; B, B'; C; x, y)\notag\\
 &=\varepsilon(xy)BB'(-1)\sum_{u, v\in \mathbb{F}_q}B(u)B'(v)C\overline{BB'}(1-u-v)\overline{A}(1-ux)\overline{A'}(1-vy).
\end{align}
Many transformation identities are also established for the finite field Appell series analogous to those satisfied by the classical Appell series $F_1, F_2$, and $F_3$ \cite{bing-he-2, bing-he-1, li-li-mao}.
\par The finite field analogues of Appell series introduced in \cite{bing-he-2, bing-he-1, li-li-mao} are in the spirit of Greene's finite field hypergeometric functions \cite{greene}.
Greene used the integral representation of classical hypergeometric series to introduce finite field hypergeometric functions, and is essentially defined using Jacobi sums.
There are other finite field analogues of classical hypergeometric series. For example, see \cite{katz, mccarthy3, FL}.
In \cite{BST}, the authors with Saikia study finite field analogues of Appell series in the spirit of the finite field hypergeometric series defined by McCarthy in \cite{mccarthy3}.
For a multiplicative character $\chi$, let $g(\chi)$ denote the Gauss sum as defined in Section 2. For $A_0, A_1, \ldots A_n, B_1, B_2, \ldots, B_n\in \widehat{\mathbb{F}_q^{\times}}$, 
the McCarthy's finite field hypergeometric function ${_{n+1}F}^{\ast}_n$ is given by
\begin{align}\label{mcCarty's_defn}
&{_{n+1}F}_n\left(\begin{array}{cccc}
                A_0, & A_1, & \ldots, & A_n\\
                 & B_1, & \ldots, & B_n
              \end{array}\mid x \right)^{\ast}\notag\\
              &=\frac{1}{q-1}\sum_{\chi\in \widehat{\mathbb{F}_q^{\times}}}\prod_{i=0}^{n}\frac{g(A_i\chi)}{g(A_i)}\prod_{j=1}^n\frac{g(\overline{B_j\chi})}{g(\overline{B_j})}g(\overline{\chi})
\chi(-1)^{n+1}\chi(x).
\end{align} 
McCarthy's function is defined purely in terms of Gauss sums. Since Greene's function is defined using Jacobi sums, often it is necessary
to impose conditions on the parameters to relate the Jacobi sums to the required product of Gauss
sums. But, McCarthy's function does not need such conditions. Unlike to Greene's function, one can interchange the positions of any two of the parameters $A_i$'s or $B_i$'s
in the McCarthy's function as in the case of classical hypergeometric series. With this motivation, the authors with Saikia \cite{BST} define three functions 
$F_1^{\ast}, F_{2}^{\ast}$, and $F_3^{\ast}$ as finite field analogues of the Appell series $F_1, F_2$, and $F_3$, respectively.
The classical Appell series are defined after multiplying two $_{2}F_1$-hypergeometric series and then arranging the products of the rising factorials in some order. 
Considering products of 
McCarthy's ${_2F_1}$-hypergeometric functions, the functions $F_1^{\ast}, F_{2}^{\ast}$, and $F_3^{\ast}$ are introduced as finite field analogues of Appell series.
Let $A, A', B, B',C, C'$ be multiplicative characters on $\mathbb{F}_q$. For $x, y\in\mathbb{F}_q$, the functions $F_1^{\ast}, F_{2}^{\ast}$, and $F_3^{\ast}$ are 
defined as follows \cite{BST}.
%%%%%%%%%%%%%%%%%%%%%%%%%%%%%%%%%%%%%%%%%%%%%%%%%%%%%
\begin{align}
\label{def1}
 &F_{1}(A;B,B^{\prime};C;x,y)^{*} \notag\\
 &= \frac{1}{(q-1)^2} \sum_{\psi,\chi \in\widehat{\mathbb{F}_q^{\times}}}\frac{g(A\chi\psi)g(B\psi)g(B^{\prime}\chi)
 g(\overline{C\chi\psi})g(\overline{\psi})g(\overline{\chi})}{g(A)g(B)g(B^{\prime})g(\overline{C})}\psi(x)\chi(y);\\
 \label{def2}
 &F_{2}(A;B,B^{\prime};C,C^{\prime};x,y)^{*}\notag\\
 &= \frac{1}{(q-1)^2} \sum_{\psi, \chi\in\widehat{\mathbb{F}_q^{\times}}}
\frac{g(A\chi\psi)g(B\psi)g(B^{\prime}\chi)g(\overline{C\psi})
g(\overline{C^{\prime}\chi})g(\overline{\psi})g(\overline{\chi})}{g(A)g(B)g(B^{\prime})g(\overline{C})g(\overline{C^{\prime}})}\psi(x)\chi(y);\\
\label{def3}
&F_{3}(A,A^{\prime};B,B^{\prime};C;x,y)^{*}\notag\\
&= \frac{1}{(q-1)^2} \sum_{\psi, \chi\in\widehat{\mathbb{F}_q^{\times}}}\frac{g(A\psi)g(A^{\prime}\chi)g(B\psi)g(B^{\prime}\chi)g(\overline{C\chi\psi})
 g(\overline{\psi})g(\overline{\chi})}{g(A)g(A^{\prime})g(B)g(B^{\prime})g(\overline{C})}\psi(x)\chi(y).
\end{align}  
We know that the rising factorial $(a)_k$ can be expressed as $(a)_k=\dfrac{\Gamma(a+k)}{\Gamma(a)}$, where $\Gamma(\cdot)$ is the gamma function. Since the Gauss sum is the finite 
field analogue of the gamma function and Appell series are defined using products of rising factorials, it seems to be more appropriate to define finite field analogues of Appell 
series using Gauss sums.
Though the finite field Appell series $F_1$, $F_2$ and $F_3$ are defined using integral
representations of Appell series in \cite{li-li-mao, bing-he-1, bing-he-2}, however it is shown in \cite{BST} that they are closely related to the above functions 
$F_1^{\ast}$, $F_2^{\ast}$, and $F_3^{\ast}$.
\par We now state the following double integral representation of the Appell series $F_4$ from \cite{chaundy}.
\begin{align}\label{int-rep-f4}
&F_4(a; b; c, c'; x(1-y), y(1-x))\notag\\
&=\frac{\Gamma(c)\Gamma(c')}{\Gamma(a)\Gamma(b)\Gamma(c-a)\Gamma(c'-b)}\int_{0}^{1}\int_{0}^{1}[u^{a-1}v^{b-1}(1-u)^{c-a-1}(1-v)^{c'-a-1}\notag\\
&\hspace{1cm}\times (1-ux)^{a-c-c'+1}(1-vy)^{b-c-c'+1}(1-ux-vy)^{c+c'-a-b-1}]dudv.
\end{align}
The above integral representation of $F_4$ is more complicated than the integral representations of $F_1, F_2$ and $F_3$.
Therefore, it is not straightforward to find an appropriate
finite field analogue of $F_4$ using the double integral representation \eqref{int-rep-f4}. In the spirit of $F_1^{\ast}$, $F_2^{\ast}$ and $F_3^{\ast}$, using Gauss sums we define
\begin{align}\label{f4-star}
& F_{4}(A;B;C,C^{\prime};x,y)^{*}\notag\\
& = \frac{1}{(q-1)^2} \sum_{\psi, \chi\in\widehat{\mathbb{F}_q^{\times}}}
 \frac{g(A\chi\psi)g(B\chi\psi)g(\overline{C\psi})g(\overline{C^{\prime}\chi})
 g(\overline{\psi})g(\overline{\chi})}{g(A)g(B)g(\overline{C})g(\overline{C^{\prime}})}\psi(x)\chi(y).
\end{align}
\par In this article we establish the function $F_4^{\ast}$ as a finite field analogue of the Appell series $F_4$ by proving results over finite fields analogous to classical results satisfied by $F_4$. 
For example, we prove the following result.
\begin{theorem}\label{MT1}
 Let $A, B, C, C'\in \widehat{\mathbb{F}_q^{\times}}$. For $x, y\in \mathbb{F}_q$ such that $x, y\neq 1$ we have
 \begin{align}
  &\overline{A}(1-x)\overline{B}(1-y)F_{4}\left(A;B;C,C^{\prime};\frac{-x}{(1-x)(1-y)},\frac{-y}{(1-x)(1-y)}\right)^{*}\notag\\
  &=\frac{1}{(q-1)^2}\sum_{\chi, \psi\in \widehat{\mathbb{F}_q^{\times}}}{_{2}}F_1\left(\begin{array}{cc}
                \overline{\chi}, & A\psi\\
                 & C'
              \end{array}\mid 1 \right)^{\ast}{_{2}}F_1\left(\begin{array}{cc}
                \overline{\psi}, & B\chi\\
                 & C
              \end{array}\mid 1 \right)^{\ast}\notag\\
              &\hspace{3cm}\times \frac{g(A\psi)g(B\chi)g(\overline{\chi})g(\overline{\psi})}{g(A)g(B)}\psi(-x)\chi(-y).\notag
 \end{align}
\end{theorem}
The above result is a finite field analogue of the following identity \cite{rainville} satisfied by the Appell series $F_4$:
\begin{align}
 &(1-x)^{-a}(1-y)^{-b}F_4\left(a; b; c, c'; \frac{-x}{(1-x)(1-y)}, \frac{-y}{(1-x)(1-y)}\right)\notag\\
 &=\sum_{n, k=0}^{\infty}F(-n, a+k; c'; 1) F(-k, b+n; c; 1)\frac{(a)_k(b)_n}{k!n!}x^ky^n.\notag
\end{align} 
\par We now state a result where the classical Appell series $F_4$ is expressed as a product of two ${_2}F_1$-classical hypergeometric series \cite[Theorem 84, p. 269]{rainville}. 
If neither $c$ nor $(1-c+a+b)$ is zero or a negative integer, then
\begin{align}\label{identity-f4-1}
 &F_4\left(a; b; c, 1-c+a+b; \frac{-x}{(1-x)(1-y)}, \frac{-y}{(1-x)(1-y)}\right)\notag\\
 &=F\left(a, b; c; \frac{-x}{1-x}\right)F\left(a, b; 1-c+a+b; \frac{-y}{1-y}\right).
\end{align}
We prove the following result which is a finite field analogue of \eqref{identity-f4-1}. Let $\delta$ be defined on $\mathbb{F}_q$ by $\delta(0)=1$ and $\delta(x)=0$ for $x\neq 0$.
\begin{theorem}\label{MT2}
Let $A, B, C \in \widehat{\mathbb{F}_q^{\times}}$ be such that $A, B, C\neq \varepsilon, $ and $B\neq C$. 
For $x, y\in \mathbb{F}_q$ such that $x, y\neq 1$  we have
 \begin{align}
  &F_{4}\left(A;B;C,AB\overline{C};\frac{-x}{(1-x)(1-y)},\frac{-y}{(1-x)(1-y)}\right)^{*}\notag\\
  &={_{2}}F_1\left(\begin{array}{cc}
                A, & B\\
                 & C
              \end{array}\mid -\frac{x}{1-x} \right)^{\ast}{_{2}}F_1\left(\begin{array}{cc}
                A, & B\\
                 & AB\overline{C}
              \end{array}\mid -\frac{y}{1-y} \right)^{\ast}\notag\\
              & + \frac{q-1}{q}\overline{A}\left(\frac{x}{x-1}\right)\overline{B}\left(\frac{y}{y-1}\right)
              -\frac{q^2AC(-1)\overline{B}C(y)A(1-x)B(1-y)}{g(A)g(B)g(\overline{C})g(\overline{AB}C)}\delta(1-xy).\notag
 \end{align}
\end{theorem}
In addition, if $xy\neq 1$ and $A\neq C$, then we have 
\begin{align}
&F_{4}\left(A;B;C,AB\overline{C};\frac{-x}{(1-x)(1-y)},\frac{-y}{(1-x)(1-y)}\right)^{*}\notag\\
  &={_{2}}F_1\left(\begin{array}{cc}
                A, & B\\
                 & C
              \end{array}\mid -\frac{x}{1-x} \right)^{\ast}{_{2}}F_1\left(\begin{array}{cc}
                A, & B\\
                 & AB\overline{C}
              \end{array}\mid -\frac{y}{1-y} \right)^{\ast}.\notag 
\end{align}

\par We next consider the following identity from \cite{bailey-1} connecting the classical Appell series $F_4$ and $F_1$.
\begin{align}\label{identity-f4-2}
 &F_4\left(a; b; c, b; \frac{-x}{(1-x)(1-y)}, \frac{-y}{(1-x)(1-y)}\right)\notag\\
 &=(1-x)^{a}(1-y)^{a}F_1\left(a; c-b; 1+a-c; c; x, xy\right).
\end{align}
In the following theorem we give a finite field analogue of \eqref{identity-f4-2}.
\begin{theorem}\label{MT3}
Let $A, B, C \in \widehat{\mathbb{F}_q^{\times}}$ be such that $B \neq \varepsilon$ and $A\neq B\neq C\neq A$. For $x, y\in \mathbb{F}_q^{\times}$ such that $x,y\neq 1$ we have
 \begin{align}
  &F_{4}\left(A;B;C,B;\frac{-x}{(1-x)(1-y)},\frac{-y}{(1-x)(1-y)}\right)^{*}\notag\\
  &=A\left((1-x)(1-y)\right) F_{1}(A;\overline{B}C,A\overline{C};\overline{C};x,xy)^{*}-\frac{g(B)g(A\overline{B})}{q\cdot g(A)}\overline{B}(y)B\left((1-x)(1-y)\right).\notag
 \end{align}
\end{theorem}
%%%%%%%%%%%%%%%%%%%%%%%%%%%%%%%%%%%%%%%%%%%%%%%%%%%%%%%%%%%%%%%%%%%%%%%%%%%%%%%%%%%%%%%%%%%%%%%%%%%%%%%%%%%%
%%%%%%%%%%%%%%%%%%%%%%%%%%%%%%%%%%%%%%%%%%%%%%%%%%%%%%%%%%%%%%%%%%%%%%%%%%%%%%%%%%%%%%%%%%%%%%%%%%%%%%%%%%%%%
%%%%%%%%%%%%%%%%%%%%%%%%%%%%%%%%%%%%%%%%%%%%%%%%%%%%%%%%%%%%%%%%%%%%%%%%%%%%%%%%%%%%%%%%%%%%%%%%%%%%%%%%%%%%%
\section{Notations and Preliminaries} We first recall some definitions and results from \cite{greene}. 
Let $\delta$ denote the function on multiplicative characters defined by
$$\delta(A)=\left\{
              \begin{array}{ll}
                1, & \hbox{if $A$ is the trivial character;} \\
                0, & \hbox{otherwise.}
              \end{array}
            \right.
$$
We also denote by $\delta$ the function defined on $\mathbb{F}_q$ by 
$$\delta(x)=\left\{
              \begin{array}{ll}
                1, & \hbox{if $x=0$;} \\
                0, & \hbox{if $x\neq 0$.}
              \end{array}
            \right.
$$
For multiplicative characters $A$ and $B$ on $\mathbb{F}_q$, the binomial coefficient ${A \choose B}$ is defined by
\begin{align}\label{b1}
{A \choose B}:=\frac{B(-1)}{q}J(A,\overline{B})=\frac{B(-1)}{q}\sum_{x \in \mathbb{F}_q}A(x)\overline{B}(1-x),
\end{align}
where $J(A, B)$ denotes the usual Jacobi sum. Binomial coefficients of characters possess many interesting properties. For example, we have
%\begin{align}\label{b2}
%{A\choose B}={A\choose A\overline{B}};
%\end{align}
\begin{align}\label{b3}
{A\choose B}=B(-1){B\overline{A}\choose B}.
\end{align}
%\begin{align}\label{b4}
%{A\choose B}=AB(-1){\overline{B}\choose \overline{A}}
%\end{align}
%and
%\begin{align}\label{b5}
%{A\choose \varepsilon}={A\choose A}=\frac{-1}{q}+\frac{q-1}{q}\delta(A).
%\end{align}
The following are character sum analogues of the binomial theorem \cite{greene}.
For any $A, B\in\widehat{\mathbb{F}_q^{\times}}$ and $x\in\mathbb{F}_q$ we have
\begin{align}\label{binom1}
\overline{A}(1-x)=\delta(x)+\frac{q}{q-1}\sum_{\chi\in\widehat{\mathbb{F}_q^{\times}}}{A\chi \choose \chi}\chi(x),\\
\label{binom2}
\overline{B}(x)\overline{A}B(1-x)=\frac{q}{q-1}\sum_{\chi\in\widehat{\mathbb{F}_q^{\times}}}{A\chi\choose B\chi}\chi(x).
\end{align}
Multiplicative characters satisfy the following orthogonal relation.
For $x\in\mathbb{F}_q$, we have
\begin{align}\label{g5}
 \sum_{\chi\in \widehat{\mathbb{F}_q^{\times}}}\chi(x) = (q-1)\delta(1-x).
\end{align}
%%%%%%%%%%%%%%%%%%%%%%%%%%%%%%%%%%%%%%%%%%%%%%%%%%%%%%%%%%%%%%%%%%%%%%%%%%%%%%%%%
%Properties of gauss sums%%%%%%%%%%%%%%%%%
\par We next recall some properties of Gauss and Jacobi sums. For further details, see \cite{evans}. Let $\zeta_p$ be a fixed primitive $p$-th root of unity
in ${\mathbb{C}}$. The trace map $\text{tr}: \mathbb{F}_q \rightarrow \mathbb{F}_p$ is given by
\begin{align}
\text{tr}(\alpha)=\alpha + \alpha^p + \alpha^{p^2}+ \cdots + \alpha^{p^{r-1}}.\notag
\end{align}
Then the additive character
$\theta: \mathbb{F}_q \rightarrow \mathbb{C}$ is defined by
\begin{align}
\theta(\alpha)=\zeta_p^{\text{tr}(\alpha)}.\notag
\end{align}
For $\chi \in \widehat{\mathbb{F}_q^\times}$, the \emph{Gauss sum} is defined by
\begin{align}
g(\chi):=\sum\limits_{x\in \mathbb{F}_q}\chi(x)\theta(x).\notag
\end{align}
It is easy to show that $g(\varepsilon)=-1$. 
\begin{lemma}\emph{(\cite[Eq. 1.12]{greene}).}\label{g1}
For $\chi\in \widehat{\mathbb{F}_q^{\times}}$ we have
$$g(\chi)g(\overline{\chi})=q\cdot \chi(-1)-(q-1)\delta(\chi).$$
\end{lemma}
The following lemma gives a relation between Gauss and Jacobi sums.
\begin{lemma}\emph{(\cite[Eq. 1.14]{greene}).}\label{gj1} For $A,B\in\widehat{\mathbb{F}_q^{\times}}$ we have
\begin{align}
J(A,B)=\frac{g(A)g(B)}{g(AB)}+(q-1)B(-1)\delta(AB).\notag
\end{align}
\end{lemma}
The following result is due to McCarthy.
\begin{lemma}\emph{(\cite[Th. 2.2]{mccarthy3}).}\label{g2}
For $A,B,C,D\in\widehat{\mathbb{F}_q^{\times}}$ we have
 \begin{align}
  &\frac{1}{q-1}\sum_{\chi\in\widehat{\mathbb{F}_q^{\times}}}g(A\chi)g(B\chi)g(C\overline{\chi})g(D\overline{\chi})\notag\\
  &\hspace{1cm}= \frac{g(AC)g(AD)g(BC)g(BD)}{g(ABCD)} + q(q-1)AB(-1)\delta(ABCD).\notag
 \end{align}
\end{lemma}
We now prove two lemmas which will be used to prove our main results.
\begin{lemma}\label{g3}
 For $A,B, C\in\widehat{\mathbb{F}_q^{\times}}$ we have
 \begin{align}
 &{_{2}}F_1\left(\begin{array}{cc}
                A, & B\\
                 & C
              \end{array}\mid 1 \right)^{\ast} = \frac{g(A\overline{C})g(B\overline{C})}{g(\overline{C})g(AB\overline{C})}
              +\frac{q(q-1)AB(-1)}{g(A)g(B)g(\overline{C})}\delta(AB\overline{C}).\notag
 \end{align}
\end{lemma}
\begin{proof}
 The proof follows directly by using  \eqref{mcCarty's_defn} and Lemma \ref{g2}.
\end{proof}
\begin{lemma}\emph{\label{g7}}
 Let $A \in \widehat{\mathbb{F}_q^{\times}}$ and $x\in\mathbb{F}_q$. For $x\neq 0, 1$ we have
 \begin{align}
  \overline{A}(1-x) = \frac{1}{q-1}\sum_{\chi\in \widehat{\mathbb{F}_q^{\times}}}\frac{g(A\chi)g(\overline{\chi})}{g(A)}\chi(-x).\notag
 \end{align}
\begin{proof}
 From \eqref{b1} and \eqref{binom1} we have
\begin{align}
 \overline{A}(1-x)=\frac{q}{q-1}\sum_{\chi\in \widehat{\mathbb{F}_q^{\times}}}{A\chi \choose \chi}\chi(x)=\frac{1}{q-1}\sum_{\chi\in \widehat{\mathbb{F}_q^{\times}}}J(A\chi, \overline{\chi})~\chi(-x).\notag
\end{align}
Lemma \ref{gj1} yields
\begin{align}
 \overline{A}(1-x)&=\frac{1}{q-1}\sum_{\chi\in \widehat{\mathbb{F}_q^{\times}}}\frac{g(A\chi)g(\overline{\chi})}{g(A)}\chi(-x)+\delta(A)\sum_{\chi\in \widehat{\mathbb{F}_q^{\times}}}\chi(x)\notag\\
 &=\frac{1}{q-1}\sum_{\chi\in \widehat{\mathbb{F}_q^{\times}}}\frac{g(A\chi)g(\overline{\chi})}{g(A)}\chi(-x).\notag
\end{align}
Here, we use the fact that $\displaystyle\sum_{\chi\in \widehat{\mathbb{F}_q^{\times}}}\chi(x)=0$ when $x\neq 1$.
\end{proof}
\end{lemma}
%%%%%%%%%%%%%%%%%%%%%%%%%%%%%%%%%%%%%%%%%%%%%%%%%%%%%%%%%%%%%%%%%%%%%%%%%%%5
We now find certain special values and transformation identities of the McCarthy's finite field hypergeometric function which will be used to prove our main results.
For this we need to use the relation between the finite field hypergeometric functions defined by Greene and McCarthy. Greene \cite{greene}
defined a finite field analogue of the hypergeometric series $F(a, b; c; x)$ using its integral representation. 
Let $A, B, C$ be multiplicative characters on $\mathbb{F}_q$. Then Greene's $_{2}F_1$-finite field
hypergeometric series is defined as
\begin{align}
{_{2}}F_1\left(\begin{array}{cc}
                A, & B\\
                 & C
              \end{array}\mid x \right)=\varepsilon(x)\frac{BC(-1)}{q}\sum_{y\in \mathbb{F}_q}B(y)\overline{B}C(1-y)\overline{A}(1-xy).\notag
\end{align}
Greene expressed the above ${_{2}}F_1$-finite field hypergeometric series in terms of binomial coefficients \cite[Theorem 3.6]{greene} as given below.
\begin{align}
{_{2}}F_1\left(\begin{array}{cc}
                A, & B\\
                 & C
              \end{array}\mid x \right)=\frac{q}{q-1}\sum_{\chi\in \widehat{\mathbb{F}_q^\times}}{A\chi \choose \chi}{B\chi \choose C\chi}\chi(x).\notag
\end{align}
In general, for positive integer $n$, Greene \cite{greene} defined the ${_{n+1}}F_n$- finite field hypergeometric series over $\mathbb{F}_q$ by
\begin{align}
{_{n+1}}F_n\left(\begin{array}{cccc}
                A_0, & A_1, & \ldots, & A_n\\
                 & B_1, & \ldots, & B_n
              \end{array}\mid x \right)
              =\frac{q}{q-1}\sum_{\chi\in \widehat{\mathbb{F}_q^\times}}{A_0\chi \choose \chi}{A_1\chi \choose B_1\chi}
              \cdots {A_n\chi \choose B_n\chi}\chi(x),\notag
\end{align}
where $A_0, A_1,\ldots, A_n$ and $B_1, B_2,\ldots, B_n$ are multiplicative characters on $\mathbb{F}_q$.
%%%%%%%%%%%%%%%%%%%%%%%%%%%%%%%%%%%%%%%%%%%%%%%%%%%%%%%%%%%%%%%%%%%%%%%%%%%%%%%%%%%%%%%%%%%%%%%%%%%%%%%%%%%%%%%
\par The following proposition relates the two finite field hypergeometric functions defined by Greene in \cite{greene} and McCarthy in \cite{mccarthy3} 
under certain conditions on the parameters.
\begin{proposition}\cite[Proposition 2.5]{mccarthy3}\label{prop-300}
If $A_0\neq\varepsilon$ and $A_i\neq B_i$ for $1\leq i\leq n$, then for $x\in \mathbb{F}_q$ we have
\begin{align}
&{_{n+1}F}_n\left(\begin{array}{cccc}
                A_0, & A_1, & \ldots, & A_n\\
                 & B_1, & \ldots, & B_n
              \end{array}\mid x \right)^{\ast}\notag\\
              &=\left[\prod_{i=1}^n{A_i \choose B_i}^{-1}\right]{_{n+1}F}_n\left(\begin{array}{cccc}
                A_0, & A_1, & \ldots, & A_n\\
                 & B_1, & \ldots, & B_n
              \end{array}\mid x \right).\notag
              \end{align}
\end{proposition}
We will also need the following two special cases which are not included in the above result. 
The proofs of these cases are easy and follow directly by using \eqref{b1}, \eqref{binom1}, \eqref{binom2}, Lemma \ref{g1}, Lemma \ref{gj1}
and the fact that $g(\varepsilon)=-1$.\\
Case 1: Let $A_0=\varepsilon$ and $A_1\neq B_1$. Then for $x\neq 0$ we have
\begin{align}\label{rel-1}
{_2F_1}\left(\begin{array}{cc}
               \varepsilon, & A_1 \\
               ~ & B_1
             \end{array}\mid x
\right)^{\ast}=1-\overline{B_1}(x){A_1\choose B_1}^{-1}\overline{A_1}B_1(1-x).
\end{align}
Case 2: Let $A_0\neq\varepsilon$, $A_1=B_1\neq\varepsilon$. Then for $x\neq 0$ we have
\begin{align}\label{rel-2}
{_2F_1}\left(\begin{array}{cc}
               A_0, & A_1 \\
               ~ & A_1
             \end{array}\mid x
\right)^{\ast}=-\overline{A_1}(x){A_0\overline{A_1}\choose \overline{A_1}}+\overline{A_0}(1-x).
\end{align}
\begin{lemma}\emph{\label{rel-3}}
 Let $A, B, C \in \widehat{\mathbb{F}_q^{\times}}$ be such that $A, B\neq \varepsilon$ and $ B\neq C$. For $x \in \mathbb{F}_q$ such that $x \neq 1$ we have
 \begin{align}
  {_{2}}F_1\left(\begin{array}{cc}
                A, & B\\
                 & C
              \end{array}\mid x \right)^{\ast}= \overline{A}(1-x){_{2}}F_1\left(\begin{array}{cc}
                                                                               A, & \overline{B}C\\
                                                                                & C
                                                                             \end{array}\mid -\frac{x}{1-x} \right)^{\ast}.\notag
 \end{align}
\end{lemma}
\begin{proof}
 By Proposition \ref{prop-300} we have
\begin{align}
 {_{2}}F_1\left(\begin{array}{cc}
                A, & B\\
                 & C
              \end{array}\mid x \right)^{\ast}&= {B \choose C}^{-1}{_{2}}F_1\left(\begin{array}{cc}
                                                                                      A, & B\\
                                                                                           & C
                                                                             \end{array}\mid x \right)\notag\\
&= {B \choose C}^{-1}C(-1)\overline{A}(1-x){_{2}}F_1\left(\begin{array}{cc}
                                                          A, & \overline{B}C\\
                                                           & C
                                                       \end{array}\mid -\frac{x}{1-x} \right).\notag \\
&=\overline{A}(1-x){_{2}}F_1\left(\begin{array}{cc}
                                  A, & \overline{B}C\\
                                     & C
                              \end{array}\mid -\frac{x}{1-x} \right)^{\ast}.\notag                                       
\end{align}
The above equalities follow from \cite[Theorem 4.4 (ii)]{greene}, \eqref{b3} and Proposition \ref{prop-300}.
\end{proof}
\begin{lemma}\emph{\label{rel-4}}
Let $A, B, C \in \widehat{\mathbb{F}_q^{\times}}$ be such that $A, B\neq \varepsilon$ and $ B\neq C$, $A\neq C$. For $x \in \mathbb{F}_q$ such that $x \neq 1$ we have
\begin{align}
  {_{2}}F_1\left(\begin{array}{cc}
                A, & B\\
                 & C
              \end{array}\mid x \right)^{\ast}= \overline{AB}C(1-x){_{2}}F_1\left(\begin{array}{cc}
                                                                               C\overline{A}, & \overline{B}C\\
                                                                                & C
                                                                             \end{array}\mid x \right)^{\ast}.\notag
 \end{align} 
\end{lemma}
\begin{proof}
By Proposition \ref{prop-300} we have
\begin{align}
{_{2}}F_1\left(\begin{array}{cc}
                A, & B\\
                 & C
              \end{array}\mid x \right)^{\ast}&= {B \choose C}^{-1}  {_{2}}F_1\left(\begin{array}{cc}
                                                                                      A, & B\\
                                                                                        & C
                                                                               \end{array}\mid x \right) \notag\\
 &={B \choose C}^{-1} C(-1) \overline{AB}C(1-x){_{2}}F_1\left(\begin{array}{cc}
                                                                C\overline{A}, & C\overline{B}\\
                                                                 & C
                                                            \end{array}\mid x \right) \notag   \\
&= \overline{AB}C(1-x){_{2}}F_1\left(\begin{array}{cc}
                                 C\overline{A}, & C\overline{B}\\
                                       & C
                                 \end{array}\mid x \right)^{\ast}. \notag                                                   
\end{align}
We note that the above equalities follow from \cite[Theorem 4.4 (iv)]{greene}, \eqref{b3} and Proposition \ref{prop-300}.
\end{proof}
\begin{lemma}\emph{\label{rel-5}}
Let $A, B \in \widehat{\mathbb{F}_q^{\times}}$ be such that $A\neq \varepsilon$ and $ A\neq B$. For $x \in \mathbb{F}_q$ such that $x \neq 1$ we have
\begin{align}
 {_2F_1}\left(\begin{array}{cc}
               A, & B \\
                & A
             \end{array}\mid x
\right)^{\ast}=\varepsilon(x)\overline{B}(1-x)-\frac{1}{q}{B\choose A}^{-1}\overline{A}(-x).\notag
\end{align}
\end{lemma}
\begin{proof}
The proof follows directly by using Proposition \ref{prop-300} and \cite[Corollary 3.16 (iii)]{greene}.
\end{proof}

%%%%%%%%%%%%%%%%%%%%%%%%%%%%
%%%%%%%%%%%%%%%%%%%%%%%%%%%%%
%%%%%%%%%%%%%%%%%%%%%%%%%%%%%%%%%%%%%%%%%%%%%%%%%%%%%%%%%%%%%%%%%%%%%%%%%%%%%%%%%%%%%%%%%%%%%%%%%
\section{Proofs of Theorem \ref{MT1} and Theorem \ref{MT2}}
We write $\displaystyle\sum_{\chi}$ to denote the sum over all multiplicative characters of $\mathbb{F}_q$.
\begin{proof}[Proof of Theorem \ref{MT1}]
The result holds trivially if either $x=0$ or $y=0$. Therefore, we assume that both $x$ and $y$ are nonzero. 
From \eqref{f4-star} and then using Lemma \ref{g7} we have
\begin{align}
L&:=\overline{A}(1-x)\overline{B}(1-y)F_{4}\left(A;B;C,C^{\prime};\frac{-x}{(1-x)(1-y)},\frac{-y}{(1-x)(1-y)}\right)^{*}\notag\\
&=\frac{1}{(q-1)^2}\sum_{\psi, \chi}\frac{g(A\chi\psi)g(B\chi\psi)g(\overline{C\psi})g(\overline{C^{\prime}\chi})
 g(\overline{\psi})g(\overline{\chi})}{g(A)g(B)g(\overline{C})g(\overline{C^{\prime}})}\psi(-x)\chi(-y)\notag\\
 &\hspace{3cm}\times\overline{A\chi\psi}(1-x)\overline{B\chi\psi}(1-y)\notag\\
 &=\frac{1}{(q-1)^4}\sum_{\psi, \chi, \eta, \lambda}\frac{g(A\chi\psi\eta)g(B\chi\psi\lambda)g(\overline{C\psi})g(\overline{C^{\prime}\chi})
 g(\overline{\psi})g(\overline{\chi})g(\overline{\eta})g(\overline{\lambda})}{g(A)g(B)g(\overline{C})g(\overline{C^{\prime}})}\psi\eta(-x)\chi\lambda(-y).\notag
\end{align}
The change of variables $\psi \mapsto \psi\overline{\eta}$ yields
\begin{align}
 L=\frac{1}{(q-1)^4}\sum_{\psi, \chi, \eta, \lambda}\frac{g(A\chi\psi)g(B\chi\psi\lambda\overline{\eta})g(\overline{C\psi}\eta)g(\overline{C^{\prime}\chi})
 g(\overline{\psi}\eta)g(\overline{\chi})g(\overline{\eta})g(\overline{\lambda})}{g(A)g(B)g(\overline{C})g(\overline{C^{\prime}})}\psi(-x)\chi\lambda(-y).\notag
\end{align}
Similarly, the change of variables $\eta \mapsto \psi\overline{\eta}$ and \eqref{mcCarty's_defn} yield
\begin{align}\label{mt1-q2}
L&=\frac{1}{(q-1)^3}\sum_{\psi, \chi, \lambda}\frac{g(A\chi\psi)g(B\chi\lambda)g(\overline{C^{\prime}\chi})
 g(\overline{\chi})g(\overline{\lambda})g(\overline{\psi})}{g(A)g(B)g(\overline{C^{\prime}})}\psi(-x)\chi\lambda(-y)\notag\\
 &\hspace{3cm}\times {_2F_1}\left(\begin{array}{cc}
               \overline{\psi}, & B\chi\lambda \\
                & C
             \end{array}\mid 1\right)^{\ast}.
\end{align}
If we apply the change of variables $\chi \mapsto \chi\overline{\lambda}$ then  \eqref{mt1-q2} reduces to
\begin{align}\label{mt1-q3}
L&=\frac{1}{(q-1)^3}\sum_{\psi, \chi, \lambda}\frac{g(A\chi\overline{\lambda}\psi)g(B\chi)g(\overline{C^{\prime}\chi}\lambda)
 g(\overline{\chi}\lambda)g(\overline{\lambda})g(\overline{\psi})}{g(A)g(B)g(\overline{C^{\prime}})}\psi(-x)\chi(-y)\notag\\
 &\hspace{3cm}\times {_2F_1}\left(\begin{array}{cc}
               \overline{\psi}, & B\chi \\
                & C
             \end{array}\mid 1\right)^{\ast}.
\end{align}
Finally, if we apply the change of variables $\lambda \mapsto \chi\overline{\lambda}$ then  \eqref{mt1-q3} reduces to
\begin{align}
L&=\frac{1}{(q-1)^3}\sum_{\psi, \chi, \lambda}\frac{g(A\lambda\psi)g(B\chi)g(\overline{C^{\prime}\lambda})
 g(\overline{\chi}\lambda)g(\overline{\lambda})g(\overline{\psi})}{g(A)g(B)g(\overline{C^{\prime}})}\psi(-x)\chi(-y)\notag\\
 &\hspace{3cm}\times {_2F_1}\left(\begin{array}{cc}
               \overline{\psi}, & B\chi \\
                & C
             \end{array}\mid 1\right)^{\ast}\notag\\
&=\frac{1}{(q-1)^2}\sum_{\chi, \psi}{_{2}}F_1\left(\begin{array}{cc}
                \overline{\chi}, & A\psi\\
                 & C'
              \end{array}\mid 1 \right)^{\ast}{_{2}}F_1\left(\begin{array}{cc}
                \overline{\psi}, & B\chi\\
                 & C
              \end{array}\mid 1 \right)^{\ast}\notag\\
              &\hspace{3cm}\times \frac{g(A\psi)g(B\chi)g(\overline{\chi})g(\overline{\psi})}{g(A)g(B)}\psi(-x)\chi(-y). \notag
\end{align}
This completes the proof of the theorem.
\end{proof}
%%%%%%%%%%%%%%%%%%%%%%%%%%%%%%%%%%%%%%%%%%%%%%%%%%%%%%%%%%%%%%%%%%%%%%%%%%%%%%%%%%%%%%%%%%%%%%%%%%%%%%%%%%%%%%%%%%%%%%%%%%%%%
\begin{proof}[Proof of Theorem \ref{MT2}]
If $xy=0$, then the result is trivial. So, we take both $x$ and $y$ are nonzero. Now, from Theorem \ref{MT1} we have
\begin{align}
 L&:=\overline{A}(1-x)\overline{B}(1-y)F_{4}\left(A;B;C,AB\overline{C};\frac{-x}{(1-x)(1-y)},\frac{-y}{(1-x)(1-y)}\right)^{*} \notag \\
  &=\frac{1}{(q-1)^2}\sum_{\psi, \chi\in \widehat{\mathbb{F}_q^{\times}}}{_{2}}F_1\left(\begin{array}{cc}
                \overline{\chi}, & A\psi\\
                 & AB\overline{C}
              \end{array}\mid 1 \right)^{\ast}{_{2}}F_1\left(\begin{array}{cc}
                \overline{\psi}, & B\chi\\
                 & C
              \end{array}\mid 1 \right)^{\ast}\notag\\
              &\hspace{3cm}\times \frac{g(A\psi)g(B\chi)g(\overline{\chi})g(\overline{\psi})}{g(A)g(B)}\psi(-x)\chi(-y).\notag
  \end{align}
 Lemma \ref{g3} yields
\begin{align}
 L&=\frac{1}{(q-1)^2}\sum_{\psi, \chi}\left(\frac{g(\overline{AB}C\overline{\chi})g(\overline{B}C\psi)}{g(\overline{AB}C)g(C\psi\overline{B\chi})}
 +\frac{q(q-1)A\psi\chi(-1)\delta(C\psi\overline{B\chi})}{g(\overline{\chi})g(A\psi)g(\overline{AB}C)}\right) \notag \\
 &\times\left(\frac{g(\overline{C\psi})g(B\overline{C}\chi)}{g(\overline{C})g(B\chi\overline{C\psi})}
 +\frac{q(q-1)B\psi\chi(-1)\delta(B\chi\overline{C\psi})}{g(\overline{\psi})g(B\chi)g(\overline{C})}\right)
 \frac{g(A\psi)g(B\chi)g(\overline{\chi})g(\overline{\psi})\psi(-x)\chi(-y)}{g(A)g(B)}\notag \\
 \label{mt2-q1}
 &=\frac{1}{(q-1)^2}\sum_{\psi, \chi}\frac{g(\overline{AB}C\overline{\chi})g(\overline{B}C\psi)g(\overline{C\psi})g(B\overline{C}\chi)g(A\psi)g(B\chi)}
 {g(A)g(B)g(\overline{C})g(\overline{AB}C)g(C\psi\overline{B\chi})g(B\chi\overline{C\psi})}  \\
 &\hspace{3cm}\times g(\overline{\chi})g(\overline{\psi})\psi(-x)\chi(-y)+\alpha_1+\alpha_2+\alpha_3,\notag
\end{align}
where
\begin{align}
 \alpha_1 &= A(-1)\frac{q}{q-1}\sum_{\psi,\chi}\frac{g(\overline{C\psi})g(B\overline{C}\chi)g(B\chi)g(\overline{\psi})}
 {g(A)g(B)g(\overline{C})g(\overline{AB}C)g(B\overline{C\psi}\chi)}\chi(y)\psi(x)\delta(\overline{B}C\overline{\chi}\psi),\notag\\
 \alpha_2&= B(-1)\frac{q}{q-1}\sum_{\psi,\chi}\frac{g(\overline{AB}C\overline{\chi})g(\overline{B}C\psi)g(A\psi)g(\overline{\chi})}
 {g(A)g(B)g(\overline{AB}C)g(\overline{B}C\overline{\chi}\psi)g(\overline{C})}\chi(y)\psi(x)\delta(B\overline{C}\overline{\psi}\chi),\notag \\
 \alpha_3&=q^2AB(-1)\sum_{\psi,\chi}\frac{\psi(-x)\chi(-y))\delta(\overline{B}C\overline{\chi}\psi)\delta(B\overline{C}\overline{\psi}\chi)}
 {g(A)g(B)g(\overline{AB}C)g(\overline{C})}.\notag
\end{align}
The above terms are nonzero only when $\overline{\chi}\psi=B\overline{C}$. So, after putting $\overline{\chi}=B\overline{C\psi}$ and using the fact that $g(\varepsilon)=-1$,
we obtain
\begin{align}\label{mt2-q2}
 \alpha_1&= -A(-1)\frac{q}{q-1}\sum_{\psi}\frac{g(\overline{C\psi})g(C\psi)g(\overline{\psi})g(\psi)}
 {g(A)g(B)g(\overline{C})g(\overline{AB}C)}\overline{B}C(y)\psi(xy),\\
 \label{mt2-q3}
 \alpha_2&= -B(-1)\frac{q}{q-1}\sum_{\psi}\frac{g(\overline{A\psi})g(A\psi)g(B\overline{C\psi})g(\overline{B}C\psi)}
 {g(A)g(B)g(\overline{C})g(\overline{AB}C)}\overline{B}C(y)\psi(xy),\\
 \alpha_3&=\frac{q^2AB(-1)\overline{B}C(-y)}{g(A)g(B)g(\overline{AB}C)g(\overline{C})}\sum_{\psi}\psi(xy).\notag
\end{align}
In case of $\alpha_3$, \eqref{g5} yields
 \begin{align}\label{mt2-q4}
 \alpha_3=\frac{q^2(q-1)AC(-1)\overline{B}C(y)}{g(A)g(B)g(\overline{AB}C)g(\overline{C})}\delta(1-xy).
 \end{align}
Using Lemma \ref{g1} we have 
\begin{align}
 \alpha_1 = -\frac{A(-1)q}{q-1}\sum_{\psi}\frac{\left(qC\psi(-1)-(q-1)\delta(C\psi)\right)\left(q\psi(-1)-(q-1)\delta(\psi)\right)}
 {g(A)g(B)g(\overline{C})g(\overline{AB}C)}\overline{B}C(y)\psi(xy).\notag
\end{align}
Now, using  \eqref{g5} and the fact that $C\neq\varepsilon$ we have
\begin{align}\label{mt2-q13}
 \alpha_1 = -\frac{q^3AC(-1)\overline{B}C(y)\delta(1-xy)}{g(A)g(B)g(\overline{C})g(\overline{AB}C)} + \frac{q^2AC(-1)\overline{B}C(y)\overline{C}(xy)}{g(A)
 g(B)g(\overline{C})g(\overline{AB}C)} + \frac{q^2AC(-1)\overline{B}C(y)}{g(A)g(B)g(\overline{C})g(\overline{AB}C)}.
\end{align}
From \eqref{mt2-q1} we have
\begin{align}
 L&=\frac{1}{(q-1)^2}\sum_{\substack{\psi, \chi\\\overline{\chi}\psi\neq B\overline{C}}}\frac{g(\overline{AB}C\overline{\chi})g(\overline{B}C\psi)g(\overline{C\psi})
 g(B\overline{C}\chi)g(A\psi)g(B\chi)}{g(A)g(B)g(\overline{C})g(\overline{AB}C)g(C\psi\overline{B\chi})g(B\chi\overline{C\psi})}\notag \\
\hspace{2cm} &\times g(\overline{\chi})g(\overline{\psi})\psi(-x)\chi(-y)+\beta +\alpha_1+\alpha_2+\alpha_3,\notag
\end{align}
where
\begin{align}
\beta =\frac{BC(-1)}{(q-1)^2}\sum_{\psi}\frac{g(\overline{A\psi})g(A\psi)g(\overline{B}C\psi)g(B\overline{C\psi})g(\overline{C\psi})g(C\psi)g(\psi)
 g(\overline{\psi})}{g(A)g(B)g(\overline{C})g(\overline{AB}C)}\overline{B}C(y)\psi(xy).\notag
\end{align}
Using Lemma \ref{g1} on $g(C\psi\overline{B\chi})g(B\chi\overline{C\psi})$ we have
\begin{align}
 L&=\frac{BC(-1)}{q(q-1)^2}\sum_{\substack{\psi, \chi\\\overline{\chi}\psi\neq B\overline{C}}}\frac{g(\overline{AB}C\overline{\chi})g(\overline{B}C\psi)g(\overline{C\psi})
 g(B\overline{C}\chi)g(A\psi)g(B\chi)}{g(A)g(B)g(\overline{C})g(\overline{AB}C)} \notag\\
\hspace{2cm} &\times g(\overline{\chi})g(\overline{\psi})\psi(x)\chi(y)+\beta +\alpha_1+\alpha_2+\alpha_3\notag\\
\label{mt2-q8}
 &=\frac{BC(-1)}{q(q-1)^2}\sum_{\psi,\chi }\frac{g(\overline{AB}C\overline{\chi})g(\overline{B}C\psi)g(\overline{C\psi})g(B\overline{C}\chi)g(A\psi)g(B\chi)}
 {g(A)g(B)g(\overline{C})g(\overline{AB}C)}\\
 &\times g(\overline{\chi})g(\overline{\psi})\psi(x)\chi(y) + \frac{q-1}{q}\beta +\alpha_1+\alpha_2+\alpha_3.\notag
\end{align}
Employing Lemma \ref{g1} on $g(C\psi)g(\overline{C\psi})$ and $g(\psi)g(\overline{\psi})$ we have
\begin{align}\label{mt2-q9}
 \beta=\frac{q^2B(-1)}{(q-1)^2}\sum_{\psi}\frac{g(\overline{A\psi})g(A\psi)g(B\overline{C\psi})g(\overline{B}C\psi)}
 {g(A)g(B)g(\overline{C})g(\overline{AB}C)}\overline{B}C(y)\psi(xy) - \beta_1 -\beta_2 +\beta_3,
 \end{align}
where 
\begin{align}
 &\beta_1 = \frac{qBC(-1)}{q-1}\sum_{\psi}\frac{g(\overline{A\psi})g(A\psi)g(B\overline{C\psi})g(\overline{B}C\psi)\overline{B}C(y)\psi(xy)}
 {g(A)g(B)g(\overline{C})g(\overline{AB}C)}\psi(-1)\delta(C\psi), \notag \\ 
 &\beta_2 = \frac{qBC(-1)}{q-1}\sum_{\psi}\frac{g(\overline{A\psi})g(A\psi)g(B\overline{C\psi})g(\overline{B}C\psi)\overline{B}C(y)\psi(xy)}
 {g(A)g(B)g(\overline{C})g(\overline{AB}C)}C\psi(-1)\delta(\psi),\notag \\
 &\beta_3 = BC(-1)\sum_{\psi}\frac{g(\overline{A\psi})g(A\psi)g(B\overline{C\psi})g(\overline{B}C\psi)\overline{B}C(y)\psi(xy)}
 {g(A)g(B)g(\overline{C})g(\overline{AB}C)}\delta(C\psi)\delta(\psi).\notag 
\end{align}
Since $\beta_1$ is nonzero only when $\psi=\overline{C}$, after putting $\psi=\overline{C}$ we obtain
\begin{align}\label{mte2-q10}
\beta_1 &= \frac{qB(-1)g(\overline{B})g(B)g(A\overline{C})g(\overline{A}C)\overline{B}C(y)\overline{C}(xy)}{(q-1)g(A)g(B)g(\overline{C})g(\overline{AB}C)}. 
\end{align}
Using Lemma \ref{g1} and the fact that $B\neq\varepsilon$, we have
\begin{align}\label{mt2-q10}
\beta_1&= \frac{q^3 AC(-1)\overline{B}C(y)\overline{C}(xy)}{(q-1)g(A)g(B)g(\overline{C})g(\overline{AB}C)} - \overline{A}(-x)\overline{B}(-y). 
\end{align}
 Similarly $\beta_2$ is nonzero only when $\psi=\varepsilon$, and hence after putting $\psi=\varepsilon$ we obtain
\begin{align}
\beta_2 &= \frac{qB(-1)g(\overline{B}C)g(B\overline{C})g(A)g(\overline{A})\overline{B}C(y)}{(q-1)g(A)g(B)g(\overline{C})g(\overline{AB}C)} \notag \\
\label{mt2-q11}
&=\frac{q^3 AC(-1)\overline{B}C(y)}{(q-1)g(A)g(B)g(\overline{C})g(\overline{AB}C)}.
\end{align}
We note that the last equality is obtained using Lemma \ref{g1} and the fact that $A\neq\varepsilon$ and $B\neq C$. Using $C\neq\varepsilon$ we obtain $\beta_3=0.$
Putting  \eqref{mt2-q9} and \eqref{mt2-q3} into \eqref{mt2-q8} we obtain
\begin{align}
 L&=\frac{BC(-1)}{q(q-1)^2}\sum_{\psi,\chi }\frac{g(\overline{AB}C\overline{\chi})g(\overline{B}C\psi)g(\overline{C\psi})g(B\overline{C}\chi)g(A\psi)g(B\chi)}
 {g(A)g(B)g(\overline{C})g(\overline{AB}C)}\notag\\
 &\times g(\overline{\chi})g(\overline{\psi})\psi(x)\chi(y) - \frac{q-1}{q}\beta_1 - \frac{q-1}{q}\beta_2 +\alpha_1+\alpha_3.\notag
\end{align}
Multiplying both numerator and denominator by $g(B\overline{C})g(\overline{B}C)$ and rearranging the terms  we have
\begin{align}
 L&=\frac{BC(-1)}{q(q-1)^2}\sum_{\psi,\chi }\frac{g(A\psi)g(\overline{B}C\psi)g(\overline{C\psi})g(B\chi)g(B\overline{C}\chi)
 g(\overline{AB}C\overline{\chi})g(B\overline{C})g(\overline{B}C)}{g(A)g(\overline{B}C)g(\overline{C})g(B)g(B\overline{C})g(\overline{AB}C)}\notag \\
&\hspace{2cm} \times g(\overline{\chi})g(\overline{\psi})\psi(x)\chi(y) - \frac{q-1}{q}\beta_1 - \frac{q-1}{q}\beta_2 +\alpha_1+\alpha_3.\notag
\end{align}
Using Lemma \ref{g1} and the fact that $B\neq C$ we have
\begin{align}
 L&=\frac{1}{(q-1)^2}\sum_{\psi,\chi }\frac{g(A\psi)g(\overline{B}C\psi)g(\overline{C\psi})g(B\chi)g(B\overline{C}\chi)g(\overline{AB}C\overline{\chi})}
 {g(A)g(\overline{B}C)g(\overline{C})g(B)g(B\overline{C})g(\overline{AB}C)}\notag\\
&\hspace{2cm} \times g(\overline{\chi})g(\overline{\psi})\psi(x)\chi(y) - \frac{q-1}{q}\beta_1 - \frac{q-1}{q}\beta_2 +\alpha_1+\alpha_3.\notag
\end{align}
Now, \eqref{mcCarty's_defn} yields
\begin{align}
 L={_{2}}F_1\left(\begin{array}{cc}
                A, & \overline{B}C\\
                 & C
              \end{array}\mid x \right)^{\ast}{_{2}}F_1\left(\begin{array}{cc}
                B, & B\overline{C}\\
                 & AB\overline{C}
              \end{array}\mid y \right)^{\ast} - \frac{q-1}{q}\beta_1 - \frac{q-1}{q}\beta_2 +\alpha_1+\alpha_3.
\end{align}  
Using Lemma \ref{rel-3} we have
\begin{align}\label{mt2-q12}
 L&=\overline{A}(1-x)\overline{B}(1-y){_{2}}F_1\left(\begin{array}{cc}
                A, & B\\
                 & C
              \end{array}\mid \frac{-x}{1-x} \right)^{\ast}{_{2}}F_1\left(\begin{array}{cc}
                B, & A\\
                 & AB\overline{C}
              \end{array}\mid \frac{-y}{1-y} \right)^{\ast}  \\
  &\hspace{3cm} - \frac{q-1}{q}\beta_1 - \frac{q-1}{q}\beta_2 +\alpha_1+\alpha_3. \notag
\end{align}
Applying  \eqref{mt2-q4}, \eqref{mt2-q13}, \eqref{mt2-q10} and \eqref{mt2-q11} into \eqref{mt2-q12} we have
\begin{align*}
 L&=\overline{A}(1-x)\overline{B}(1-y){_{2}}F_1\left(\begin{array}{cc}
                A, & B\\
                 & C
              \end{array}\mid \frac{-x}{1-x} \right)^{\ast}{_{2}}F_1\left(\begin{array}{cc}
                B, & A\\
                 & AB\overline{C}
              \end{array}\mid \frac{-y}{1-y} \right)^{\ast}  \\
  &\hspace{3cm}  +\frac{q-1}{q}\overline{A}(-x)\overline{B}(-y)- \frac{q^2AC(-1)\overline{B}C(y)\delta(1-xy)}{g(A)g(B)g(\overline{C})g(\overline{AB}C)}. 
\end{align*}
Finally,  multiplying both sides by $A(1-x)B(1-y)$ we deduce the first identity of the theorem.
\par In addition, if we have $A\neq C$, then from \eqref{mte2-q10} we deduce that
\begin{align}
\beta_1 &= \frac{qB(-1)g(\overline{B})g(B)g(A\overline{C})g(\overline{A}C)\overline{B}C(y)\overline{C}(xy)}{(q-1)g(A)g(B)g(\overline{C})g(\overline{AB}C)}\notag\\
\label{mte2-q11}
&=\frac{q^3AC(-1)\overline{B}C(y)\overline{C}(xy)}{(q-1)g(A)g(B)g(\overline{C})g(\overline{AB}C)}.
\end{align}
We note that the last equality is obtained using Lemma \ref{g1} and the fact that $B\neq \varepsilon$ and $A\neq C$. 
Applying  \eqref{mt2-q4}, \eqref{mt2-q13}, \eqref{mt2-q11} and \eqref{mte2-q11} into \eqref{mt2-q12} we have
\begin{align*}
 L&=\overline{A}(1-x)\overline{B}(1-y){_{2}}F_1\left(\begin{array}{cc}
                A, & B\\
                 & C
              \end{array}\mid \frac{-x}{1-x} \right)^{\ast}{_{2}}F_1\left(\begin{array}{cc}
                B, & A\\
                 & AB\overline{C}
              \end{array}\mid \frac{-y}{1-y} \right)^{\ast}  \\
  &\hspace{3cm}  - \frac{q^2AC(-1)\overline{B}C(y)\delta(1-xy)}{g(A)g(B)g(\overline{C})g(\overline{AB}C)}. 
\end{align*}
Finally, multiplying both sides by $A(1-x)B(1-y)$ and using the fact that $xy\neq1$, we readily obtain the second identity of the theorem. This completes the proof.
\end{proof}
%%%%%%%%%%%%%%%%%%%%%%%%%%%%%%%%%%%%%%%%%%%%%%%%%%%%%%%%%%%%%%%%%%%%%%%%%%%%%%%%%%%%%%%%%%%%%%%%%%%%%%%%%%%%%%%%%%%%%%%%%%%%%%%%%%%%%%%%%%%%%%%%%%
%%%%%%%%%%%%%%%%%%%%%%%%%%%%%%%%%%%%%%%%%%%%%%%%%%%%%%%%%%%%%%%%%%%%%%%%%%%%%%%%%%%%%%%%%%%%
\section{Proof of Theorem \ref{MT3}}
\begin{proof}[Proof of Theorem \ref{MT3}]
From Theorem \ref{MT1} we have
\begin{align}
 L&:=\overline{A}(1-x)\overline{B}(1-y)F_{4}\left(A;B;C,B;\frac{-x}{(1-x)(1-y)},\frac{-y}{(1-x)(1-y)}\right)^{*} \notag \\
  &=\frac{1}{(q-1)^2}\sum_{\psi, \chi\in \widehat{\mathbb{F}_q^{\times}}}{_{2}}F_1\left(\begin{array}{cc}
                \overline{\chi}, & A\psi\\
                 & B
              \end{array}\mid 1 \right)^{\ast}{_{2}}F_1\left(\begin{array}{cc}
                \overline{\psi}, & B\chi\\
                 & C
              \end{array}\mid 1 \right)^{\ast}\notag\\
              &\hspace{3cm}\times \frac{g(A\psi)g(B\chi)g(\overline{\chi})g(\overline{\psi})}{g(A)g(B)}\psi(-x)\chi(-y).\notag
  \end{align}
Using Lemma \ref{g3} we have
\begin{align}
 L&=\frac{1}{(q-1)^2}\sum_{\psi, \chi}\left(\frac{g(\overline{B\chi})g(A\overline{B}\psi)}{g(\overline{B})g(A\psi\overline{B\chi})}
 +\frac{q(q-1)A\psi\chi(-1)\delta(A\psi\overline{B\chi})}{g(\overline{\chi})g(A\psi)g(\overline{B})}\right) \notag \\
 &\times\left(\frac{g(\overline{C\psi})g(B\overline{C}\chi)}{g(\overline{C})g(B\chi\overline{C\psi})}
 +\frac{q(q-1)B\psi\chi(-1)\delta(B\chi\overline{C\psi})}{g(\overline{\psi})g(B\chi)g(\overline{C})}\right)
 \frac{g(A\psi)g(B\chi)g(\overline{\chi})g(\overline{\psi})\psi(-x)\chi(-y)}{g(A)g(B)}\notag\\
 &=\frac{1}{(q-1)^2}\sum_{\psi, \chi }\frac{g(A\psi)g(A\overline{B}\psi)g(\overline{C\psi})g(B\chi)
 g(B\overline{C}\chi)g(\overline{B\chi})}{g(A)g(B)g(\overline{C})g(\overline{B})g(A\overline{B\chi}\psi)
 g(B\overline{C\psi}\chi)} \notag\\
 &\hspace{4cm}\times g(\overline{\chi})g(\overline{\psi})\psi(-x)\chi(-y) + \alpha + \beta + \gamma, \notag
\end{align}
where 
\begin{align}
\alpha&=\frac{qA(-1)}{q-1}\sum_{ \psi, \chi}\frac{g(\overline{C\psi})g(B\overline{C}\chi)g(B\chi)g(\overline{\psi})}{g(A)g(B)g(\overline{C})
g(\overline{B})g(B\chi\overline{C\psi})}\psi(x)\chi(y)\delta(A\overline{B\chi}\psi),\notag\\
\beta&=\frac{qB(-1)}{q-1}\sum_{ \psi,\chi}\frac{g(\overline{B\chi})g(A\overline{B}\psi)g(A\psi)g(\overline{\chi})}{g(A)g(B)g(\overline{C})
g(\overline{B})g(A\psi\overline{B\chi})}\psi(x)\chi(y)\delta(B\chi\overline{C\psi}),\notag\\
\gamma&=q^2AB(-1)\sum_{\psi, \chi}\frac{\psi(-x)\chi(-y)\delta(A\psi\overline{B\chi})\delta(B\chi\overline{C\psi})}{g(\overline{B})g(\overline{C})}.\notag
\end{align}
 The term $\alpha$ is nonzero only when $\overline{\chi}\psi=\overline{A}B$. After putting $\overline{\chi}=\overline{A\psi}B$  we obtain
\begin{align}
\alpha = \frac{qA(-1)}{q-1}\sum_{ \psi}\frac{g(\overline{C\psi})g(A\overline{C}\psi)g(A\psi)g(\overline{\psi})}{g(A)g(B)g(\overline{C})
g(\overline{B})g(A\overline{C})}A\overline{B}(y)\psi(xy).\notag
\end{align}
Using Lemma \ref{g1} and the fact that $B\neq\varepsilon$ we have
\begin{align}
\alpha=\frac{AB(-1)}{q-1}\sum_{ \psi}\frac{g(\overline{C\psi})g(A\overline{C}\psi)g(A\psi)g(\overline{\psi})}{g(A)g(\overline{C})g(A\overline{C})}A\overline{B}(y)\psi(xy).\notag
\end{align}
Now, multiplying both numerator and denominator by $g(\overline{A}C)$ and  then using  Lemma \ref{g1} and the fact that $A\neq C$ we have
\begin{align}\label{mt3-q2}
\alpha=\frac{BC(-1)g(\overline{A}C)}{q(q-1)}\sum_{ \psi}\frac{g(\overline{C\psi})g(A\overline{C}\psi)g(A\psi)g(\overline{\psi})}{g(A)g(\overline{C})}A\overline{B}(y)\psi(xy).
\end{align}
The term $\beta$ is nonzero only when $\overline{\psi}\chi=\overline{B}C$. After putting $\chi=\overline{B}C\psi$  we obtain
\begin{align}
\beta=\frac{qB(-1)}{q-1}\sum_{ \psi}\frac{g(\overline{C\psi})g(A\overline{B}\psi)g(A\psi)g(B\overline{C\psi})}{g(A)g(B)g(\overline{C})
g(\overline{B})g(A\overline{C})}\overline{B}C(y)\psi(xy).\notag
\end{align}
Using Lemma \ref{g1} and the fact that $B\neq\varepsilon$ we have
\begin{align}
\beta=\frac{1}{q-1}\sum_{ \psi}\frac{g(\overline{C\psi})g(A\overline{B}\psi)g(A\psi)g(B\overline{C\psi})}{g(A)g(\overline{C})
g(A\overline{C})}\overline{B}C(y)\psi(xy).\notag
\end{align} 
 Multiplying both numerator and denominator by $g(\overline{A}C)$ and then using  Lemma \ref{g1} and the fact that $A\neq C$ we have
\begin{align}\label{mt3-q3}
\beta=\frac{AC(-1)g(\overline{A}C)}{q(q-1)}\sum_{ \psi}\frac{g(\overline{C\psi})g(A\overline{B}\psi)g(A\psi)g(B\overline{C\psi})}{g(A)
g(\overline{C})}\overline{B}C(y)\psi(xy).
\end{align} 
By the fact that $A\neq C$ we obtain $\gamma=0.$
Using Lemma \ref{g1} on $g(B)g(\overline{B})$ and the fact that $B\neq\varepsilon$ we have
\begin{align}
 L&=\frac{B(-1)}{q(q-1)^2}\sum_{\psi, \chi }\frac{g(A\psi)g(A\overline{B}\psi)g(\overline{C\psi})g(B\chi)
 g(B\overline{C}\chi)g(\overline{B\chi})}{g(A)g(\overline{C})g(A\overline{B\chi}\psi)g(B\overline{C\psi}\chi)} \notag \\
 &\hspace{4cm}\times g(\overline{\chi})g(\overline{\psi})\psi(-x)\chi(-y) + \alpha + \beta. \notag
\end{align}
Again, using Lemma \ref{g1} on $g(B\chi)g(\overline{B\chi})$ we have
\begin{align}
 L&=\frac{1}{(q-1)^2}\sum_{\psi, \chi }\frac{g(A\psi)g(A\overline{B}\psi)g(\overline{C\psi})
 g(B\overline{C}\chi)}{g(A)g(\overline{C})g(A\overline{B\chi}\psi)g(B\overline{C\psi}\chi)} \notag \\
 &\hspace{3cm}\times g(\overline{\chi})g(\overline{\psi})\psi(-x)\chi(y)-\alpha_1 + \alpha + \beta, \notag
\end{align}
where
\begin{align}
\alpha_1=\frac{B(-1)}{q(q-1)}\sum_{\psi, \chi}\frac{g(A\psi)g(A\overline{B}\psi)g(\overline{C\psi})g(B\overline{C}\chi)}{g(A)g(\overline{C})g(A\overline{B\chi}\psi)
g(B\overline{C\psi}\chi)}g(\overline{\chi})g(\overline{\psi})\psi(-x)\chi(-y)\delta(B\chi). \notag
\end{align}
The term $\alpha_1$ is nonzero only when $\chi=\overline{B}$, and hence after putting $\chi=\overline{B}$ we obtain
\begin{align}
\alpha_1=\frac{\overline{B}(y)g(B)}{q(q-1)g(A)}\sum_{ \psi}g(A\overline{B}\psi)g(\overline{\psi})\psi(-x). \notag
\end{align}
Multiplying both numerator and denominator by $g(A\overline{B})$ and  then using  Lemma \ref{g7} we have
\begin{align}\label{mt3-q4}
\alpha_1=\frac{g(A\overline{B})g(B)}{qg(A)}\overline{B}(y)\overline{A}B(1-x).
\end{align}
Now, multiplying both numerator and denominator by $g(\overline{A}B\chi\overline{\psi})g(\overline{B}C\psi\overline{\chi})$ we have
\begin{align}
L&=\frac{1}{(q-1)^2}\sum_{\psi, \chi }\frac{g(A\psi)g(A\overline{B}\psi)g(\overline{C\psi})g(B\overline{C}\chi)g(\overline{A}B\chi\overline{\psi})
g(\overline{B}C\psi\overline{\chi})}{g(A)g(\overline{C})g(A\overline{B\chi}\psi)g(\overline{A}B\chi\overline{\psi})g(B\overline{C\psi}\chi)
g(\overline{B}C\psi\overline{\chi})} \notag \\
&\hspace{3cm}\times g(\overline{\chi})g(\overline{\psi})\psi(-x)\chi(y)-\alpha_1 + \alpha + \beta \notag \\
&=\frac{1}{(q-1)^2}\sum_{\substack{\psi, \chi \\ \overline{\chi}\psi \neq \overline{A}B, B\overline{C}}}\frac{g(A\psi)g(A\overline{B}\psi)
g(\overline{C\psi})g(B\overline{C}\chi)g(\overline{A}B\chi\overline{\psi})
g(\overline{B}C\psi\overline{\chi})}{g(A)g(\overline{C})g(A\overline{B\chi}\psi)g(\overline{A}B\chi\overline{\psi})g(B\overline{C\psi}\chi)
g(\overline{B}C\psi\overline{\chi})} \notag \\
&\hspace{3cm}\times g(\overline{\chi})g(\overline{\psi})\psi(-x)\chi(y)+\alpha_2 + \alpha_3 -\alpha_1 + \alpha + \beta, \notag 
\end{align}
where
\begin{align}
\alpha_2&=\frac{1}{(q-1)^2}\sum_{\substack{\psi \\ \overline{\chi} = \overline{A\psi}B}}\frac{g(A\psi)g(A\overline{B}\psi)
g(\overline{C\psi})g(A\overline{C}\psi)g(\varepsilon)g(\overline{A}C)g(\overline{A}B\overline{\psi})}{g(A)g(\overline{C})
g(\varepsilon)g(A\overline{C})g(\varepsilon)g(\overline{A}C)} \notag \\
&\hspace{3cm}\times g(\overline{\psi})A\overline{B}(y)\psi(-xy),  \notag\\
\alpha_3&=\frac{1}{(q-1)^2}\sum_{\substack{\psi \\ \overline{\chi} = B\overline{C\psi}}}\frac{g(A\psi)g(A\overline{B}\psi)
g(\overline{C\psi})g(\psi)g(\overline{A}C)g(\varepsilon)g(B\overline{C\psi})}{g(A)g(\overline{C})g(A\overline{C})
g(\varepsilon)g(\overline{A}C)g(\varepsilon)} \notag \\
&\hspace{3cm}\times g(\overline{\psi})\overline{B}C(y)\psi(-xy).  \notag
\end{align}
Using Lemma \ref{g1} on $g(A\overline{C})g(\overline{A}C)$ and the fact that $g(\varepsilon)= -1$ and $A\neq C$ we have
\begin{align}\label{mt3-q5}
\alpha_2&=-\frac{AC(-1)}{q(q-1)^2}\sum_{\substack{\psi \\ \overline{\chi} = \overline{A\psi}B}}\frac{g(A\psi)g(A\overline{B}\psi)
g(\overline{C\psi})g(A\overline{C}\psi)g(\overline{A}C)g(\overline{A}B\overline{\psi})}{g(A)g(\overline{C})}  \\
&\hspace{3cm}\times g(\overline{\psi})A\overline{B}(y)\psi(-xy).  \notag
\end{align}
Again, using Lemma \ref{g1} on $g(A\overline{B}\psi)g(\overline{A\psi}B)$ we have 
\begin{align}\label{mt3-q18}
\alpha_2&=-\frac{BC(-1)}{(q-1)^2}\sum_{\psi}\frac{g(A\psi)
g(\overline{C\psi})g(A\overline{C}\psi)g(\overline{A}C)g(\overline{\psi})}{g(A)g(\overline{C})}A\overline{B}(y)\psi(xy) + I_1,
\end{align}
where
\begin{align}
I_1&=\frac{AC(-1)}{q(q-1)}\sum_{\psi}\frac{g(A\psi)
g(\overline{C\psi})g(A\overline{C}\psi)g(\overline{A}C)g(\overline{\psi})}{g(A)g(\overline{C})}A\overline{B}(y)\psi(-xy)\delta(A\overline{B}\psi).\notag
\end{align}
The term $I_1$ is nonzero only when $\psi=\overline{A}B$. After putting $\psi=\overline{A}B$ we obtain
\begin{align}\label{mt3-q19}
I_1=\frac{BC(-1)g(B)g(A\overline{BC})g(B\overline{C})g(\overline{A}C)g(A\overline{B})\overline{A}B(x)}{q(q-1)g(A)g(\overline{C})}.
\end{align}
Using Lemma \ref{g1} on $g(A\overline{C})g(\overline{A}C)$ and the fact that $g(\varepsilon)= -1$ and $A\neq C$ we have
\begin{align}\label{mt3-q6}
\alpha_3&=-\frac{AC(-1)}{q(q-1)^2}\sum_{\substack{\psi \\ \overline{\chi} = B\overline{C\psi}}}\frac{g(A\psi)g(A\overline{B}\psi)
g(\overline{C\psi})g(\psi)g(\overline{A}C)g(B\overline{C\psi})}{g(A)g(\overline{C})} \\
&\hspace{3cm}\times g(\overline{\psi})\overline{B}C(y)\psi(-xy).  \notag
\end{align}
Employing Lemma \ref{g1} on $g(\psi)g(\overline{\psi})$ we have
\begin{align}\label{mt3-q20}
\alpha_3&=-\frac{AC(-1)}{(q-1)^2}\sum_{\psi }\frac{g(A\psi)g(A\overline{B}\psi)g(\overline{C\psi})g(\overline{A}C)g(B\overline{C\psi})}
{g(A)g(\overline{C})}\overline{B}C(y)\psi(xy)+ I_2, 
\end{align}
where
\begin{align}
 I_2= \frac{AC(-1)}{q(q-1)}\sum_{\psi }\frac{g(A\psi)g(A\overline{B}\psi)g(\overline{C\psi})g(\overline{A}C)g(B\overline{C\psi})}
{g(A)g(\overline{C})}\overline{B}C(y)\psi(-xy)\delta(\psi).\notag
\end{align}
The term $I_2$ is nonzero only when $\psi=\varepsilon$, and so after putting $\psi=\varepsilon$ we obtain
\begin{align}\label{mt3-q21}
I_2=\frac{AC(-1)g(A\overline{B})g(B\overline{C})g(\overline{A}C)\overline{B}C(y)}{q(q-1)}.
\end{align}
Using Lemma \ref{g1} on $g(A\overline{B\chi}\psi)g(\overline{A}B\chi\overline{\psi})$ and $g(B\overline{C\psi}\chi)
g(\overline{B}C\psi\overline{\chi})$ we have 
\begin{align}\label{mt3-q8}
L&=\frac{AC(-1)}{q^2(q-1)^2}\sum_{\substack{\psi, \chi \\ \overline{\chi}\psi \neq \overline{A}B, B\overline{C}}}\frac{g(A\psi)g(A\overline{B}\psi)
g(\overline{C\psi})g(B\overline{C}\chi)g(\overline{A}B\chi\overline{\psi})g(\overline{B}C\psi\overline{\chi})}{g(A)g(\overline{C})}  \\
&\hspace{3cm}\times g(\overline{\chi})g(\overline{\psi})\psi(-x)\chi(y)+ \alpha_2 + \alpha_3 -\alpha_1 + \alpha + \beta. \notag
\end{align}
 The term under summation for $\overline{\chi} = \overline{A\psi}B$ in \eqref{mt3-q8} is equal to 
\begin{align}\label{mt3-q7}
&-\frac{AC(-1)}{q^2(q-1)^2}\sum_{\substack{\psi \\ \overline{\chi} = \overline{A\psi}B}}\frac{g(A\psi)g(A\overline{B}\psi)
g(\overline{C\psi})g(A\overline{C}\psi)g(\overline{A}C)g(\overline{A}B\overline{\psi})}{g(A)g(\overline{C})}  \\
&\hspace{3cm}\times g(\overline{\psi})A\overline{B}(y)\psi(-xy).\notag
\end{align}
 Hence, applying  \eqref{mt3-q5} and \eqref{mt3-q7} into \eqref{mt3-q8}, we obtain
\begin{align}\label{mt3-q9}
L&=\frac{AC(-1)}{q^2(q-1)^2}\sum_{\substack{\psi, \chi \\ \overline{\chi}\psi \neq  B\overline{C}}}\frac{g(A\psi)g(A\overline{B}\psi)
g(\overline{C\psi})g(B\overline{C}\chi)g(\overline{A}B\chi\overline{\psi})g(\overline{B}C\psi\overline{\chi})}{g(A)g(\overline{C})}  \\
&\hspace{3cm}\times g(\overline{\chi})g(\overline{\psi})\psi(-x)\chi(y)+ \frac{q-1}{q}\alpha_2 + \alpha_3 -\alpha_1 + \alpha + \beta. \notag
\end{align}
Similarly, the term under summation for $\overline{\chi} = \overline{C\psi}B$ in \eqref{mt3-q9} is equal to
\begin{align}\label{mt3-q10}
&-\frac{AC(-1)}{q^2(q-1)^2}\sum_{\substack{\psi \\ \overline{\chi} = B\overline{C\psi}}}\frac{g(A\psi)g(A\overline{B}\psi)
g(\overline{C\psi})g(\psi)g(\overline{A}C)g(B\overline{C\psi})}{g(A)g(\overline{C})} \\
&\hspace{3cm}\times g(\overline{\psi})\overline{B}C(y)\psi(-xy).  \notag
\end{align}
Then applying  \eqref{mt3-q6} and \eqref{mt3-q10} into \eqref{mt3-q9} we obtain
\begin{align}\label{mt3-q22}
L&=\frac{AC(-1)}{q^2(q-1)^2}\sum_{\psi, \chi }\frac{g(A\psi)g(A\overline{B}\psi)g(\overline{C\psi})g(B\overline{C}\chi)g(\overline{A}B\chi\overline{\psi})
g(\overline{B}C\psi\overline{\chi})}{g(A)g(\overline{C})}  \\
&\hspace{2cm}\times g(\overline{\chi})g(\overline{\psi})\psi(-x)\chi(y)+ \frac{q-1}{q}\alpha_2 + \frac{q-1}{q}\alpha_3 -\alpha_1 + \alpha + \beta. \notag
\end{align}
 Applying  \eqref{mt3-q2}, \eqref{mt3-q18} and \eqref{mt3-q3}, \eqref{mt3-q20}  into \eqref{mt3-q22} we obtain
\begin{align}
L&=\frac{AC(-1)}{q^2(q-1)^2}\sum_{\psi, \chi }\frac{g(A\psi)g(A\overline{B}\psi)g(\overline{C\psi})g(B\overline{C}\chi)g(\overline{A}B\chi\overline{\psi})
g(\overline{B}C\psi\overline{\chi})}{g(A)g(\overline{C})}  \\
&\hspace{2cm}\times g(\overline{\chi})g(\overline{\psi})\psi(-x)\chi(y)+ \frac{q-1}{q}I_1 + \frac{q-1}{q}I_2 -\alpha_1 . \notag
\end{align}
Now, multiplying both numerator and denominator by $g(\overline{A}B\overline{\psi})g(\overline{B}C\psi)g(B\overline{C})$ and then rearranging the terms, we obtain
\begin{align}
L&=\frac{AC(-1)}{q^2(q-1)^2}\sum_{\psi, \chi}\frac{g(\overline{A\psi}B\chi)g(B\overline{C}\chi)g(\overline{B\chi}C\psi)g(\overline{\chi})g(A\psi)g(\overline{C\psi})
g(A\overline{B}\psi)}{g(\overline{A}B\overline{\psi})g(B\overline{C})g(\overline{B}C\psi)g(A)g(\overline{C})}\notag \\
&\times g(\overline{A}B\overline{\psi})g(\overline{B}C\psi)g(B\overline{C}) g(\overline{\psi})\psi(-x)\chi(y)+ \frac{q-1}{q}I_1 + \frac{q-1}{q}I_2 -\alpha_1.\notag
\end{align} 
Using \eqref{mcCarty's_defn} we have
\begin{align}
L&=\frac{AC(-1)}{q^2(q-1)}\sum_{\psi}\frac{g(A\psi)g(\overline{C\psi})g(A\overline{B}\psi)g(\overline{A}B\overline{\psi})g(\overline{B}C\psi)g(B\overline{C}) 
g(\overline{\psi})}{g(A)g(\overline{C})}\psi(-x)\notag \\
&\hspace{1cm}\times{_{2}}F_1\left(\begin{array}{cc}
                \overline{A}B\overline{\psi}, & B\overline{C}\\
                 & B\overline{C\psi}
              \end{array}\mid y \right)^{\ast} + \frac{q-1}{q}I_1 + \frac{q-1}{q}I_2 -\alpha_1.\notag\\
&=\frac{AC(-1)}{q^2(q-1)}\sum_{\substack{\psi\\ \psi\neq\varepsilon,\overline{A}B}}\frac{g(A\psi)g(\overline{C\psi})g(A\overline{B}\psi)g(\overline{A}B\overline{\psi})g(\overline{B}C\psi)g(B\overline{C}) 
g(\overline{\psi})}{g(A)g(\overline{C})}\psi(-x)\notag \\
&\hspace{1cm}\times{_{2}}F_1\left(\begin{array}{cc}
                \overline{A}B\overline{\psi}, & B\overline{C}\\
                 & B\overline{C\psi}
              \end{array}\mid y \right)^{\ast} + \beta_1 + \beta_2+ \frac{q-1}{q}I_1 + \frac{q-1}{q}I_2 -\alpha_1,\notag
\end{align}
where
\begin{align}
\beta_1 &=\frac{AC(-1)g(A\overline{B})g(\overline{A}B)g(\overline{B}C)g(B\overline{C})g(\varepsilon)}{q^2(q-1)}{_{2}}F_1\left(\begin{array}{cc}
                                                                                                                \overline{A}B, & B\overline{C}\\
                                                                                                                    & B\overline{C}
                                                                                                                   \end{array}\mid y \right)^{\ast},\notag \\ 
\beta_2 &=\frac{AC(-1)g(B)g(A\overline{BC})g(\varepsilon)g(\varepsilon)g(\overline{A}C)g(B\overline{C})g(A\overline{B})\overline{A}B(-x))}
{q^2(q-1)g(A)g(\overline{C})}{_{2}}F_1\left(\begin{array}{cc}
                \varepsilon, & B\overline{C}\\
                 & A\overline{C}
              \end{array}\mid y \right)^{\ast}. \notag                                                                                                                
\end{align}
Using Lemma \ref{g1}  and the fact that $A\neq B, B\neq C $ and $g(\varepsilon)=-1$ we obtain
\begin{align}
 \beta_1&=-\frac{1}{(q-1)}{_{2}}F_1\left(\begin{array}{cc}
                                          \overline{A}B, & B\overline{C}\\
                                          & B\overline{C}
                                       \end{array}\mid y \right)^{\ast}. \notag
\end{align}
Using  \eqref{rel-2} we have
\begin{align}\label{mt3-q11}
 \beta_1={\overline{A}C\choose\overline{B}C}\frac{\overline{B}C(y)}{(q-1)}-\frac{A\overline{B}(1-y)}{q-1}.
\end{align}              
Using \eqref{rel-1} and the fact that $g(\varepsilon)= -1$ we have
\begin{align}
\beta_2 &= \frac{BC(-1)g(B)g(A\overline{BC})g(\overline{A}C)g(A\overline{B})g(B\overline{C})\overline{A}B(x)}{q^2(q-1)g(A)g(\overline{C})} \notag \\
&-{B\overline{C}\choose A\overline{C}}^{-1}\frac{BC(-1)g(B)g(A\overline{BC})g(\overline{A}C)g(A\overline{B})g(B\overline{C})\overline{A}B(x)\overline{A}C(y)
A\overline{B}(1-y)}{q^2(q-1)g(A)g(\overline{C})}.\notag
\end{align}
Applying \eqref{b3} on the second term we obtain
\begin{align}\label{mt3-q12}
\beta_2 &= \frac{BC(-1)g(B)g(A\overline{BC})g(\overline{A}C)g(A\overline{B})g(B\overline{C})\overline{A}B(x)}{q^2(q-1)g(A)g(\overline{C})}\\
&-{A\overline{B}\choose A\overline{C}}^{-1}\frac{AB(-1)g(B)g(A\overline{BC})g(\overline{A}C)g(A\overline{B})g(B\overline{C})\overline{A}B(x)\overline{A}C(y)
A\overline{B}(1-y)}{q^2(q-1)g(A)g(\overline{C})}.\notag
\end{align}
Lemma \ref{rel-4} yields
\begin{align}\label{mt3-q24}
L&=\frac{AC(-1)}{q^2(q-1)}\sum_{\substack{\psi\\ \psi\neq\varepsilon,\overline{A}B}}\frac{g(A\psi)g(\overline{C\psi})g(A\overline{B}\psi)g(\overline{A}B\overline{\psi})
g(\overline{B}C\psi)g(B\overline{C})g(\overline{\psi})}{g(A)g(\overline{C})}\psi(-x) \\
&\hspace{1cm}\times A\overline{B}(1-y){_{2}}F_1\left(\begin{array}{cc}
                  A\overline{C}, & \overline{\psi}\\
                  & B\overline{C\psi}
              \end{array}\mid y \right)^{\ast}+ \beta_1 + \beta_2 + \frac{q-1}{q}I_1 + \frac{q-1}{q}I_2 -\alpha_1,\notag
\end{align}
Using Lemma \ref{g1} on $g(A\overline{B}\psi)g(\overline{A}B\overline{\psi})$ in \eqref{mt3-q24} and the fact that $\psi\neq \overline{A}B$ we have
\begin{align}\label{mt3-q13}
L&=\frac{BC(-1)}{q(q-1)}\sum_{\substack{\psi\\ \psi\neq\varepsilon,\overline{A}B}}\frac{g(A\psi)g(\overline{C\psi})g(\overline{B}C\psi)g(B\overline{C})
g(\overline{\psi})}{g(A)g(\overline{C})}\psi(x) \\
&\times A\overline{B}(1-y){_{2}}F_1\left(\begin{array}{cc}
                                               A\overline{C}, & \overline{\psi}\\
                                               & B\overline{C\psi}
                                           \end{array}\mid y \right)^{\ast}+\beta_1 + \beta_2 + \frac{q-1}{q}I_1 + \frac{q-1}{q}I_2 -\alpha_1.\notag
\end{align}     
If we consider the term under summation for $\psi = \varepsilon$, and then using Lemma \ref{g1}, \eqref{rel-1} and the fact that $g(\varepsilon)=-1$ we obtain 
\begin{align}\label{mt3-q14}
 I_3 = -\frac{A\overline{B}(1-y)}{q-1} + {A\overline{C} \choose B\overline{C}}^{-1}\frac{\overline{B}C(y)}{q-1}.
\end{align}
Now, using  \eqref{mt3-q11} and \eqref{mt3-q14}
we have
\begin{align}
\beta_1 - I_3 &= {\overline{A}C\choose\overline{B}C}\frac{\overline{B}C(y)}{(q-1)}-{A\overline{C} \choose B\overline{C}}^{-1}
\frac{\overline{B}C(y)}{q-1}\notag\\
&= -\frac{qAC(-1)\overline{B}C(y)}{g(\overline{A}B)g(A\overline{C})g(\overline{B}C)}.\notag
\end{align}
The last equality is obtained by using  \eqref{b1}, Lemma \ref{gj1} and Lemma \ref{g1}. If we multiplying both numerator and denominator 
by $g(B\overline{C})g(A\overline{B})g(\overline{A}C)$ and then use Lemma \ref{g1} we have
\begin{align}\label{mt3-q15}
 \beta_1 - I_3= -\frac{AC(-1)g(\overline{C}B)g(A\overline{B})g(\overline{A}C)\overline{B}C(y)}{q^2}.
\end{align}
Considering the term under summation for $\psi=\overline{A}B$ in \eqref{mt3-q13}, and then using Lemma \ref{rel-5} we have
\begin{align}\label{mt3-q16}
I_4&=BC(-1)\frac{g(B)g(A\overline{BC})g(\overline{A}C)g(A\overline{B})g(B\overline{C})\overline{A}B(x)}{q(q-1)g(A)g(\overline{C})}\\
&-AB(-1){A\overline{B}\choose A\overline{C}}^{-1}\frac{g(B)g(A\overline{BC})g(\overline{A}C)g(A\overline{B})g(B\overline{C})\overline{A}B(x)A\overline{B}(1-y)
\overline{A}C(y)}{q^2(q-1)g(A)g(\overline{C})}.\notag
\end{align}
Now, using  \eqref{mt3-q12} and \eqref{mt3-q16} we have
\begin{align}\label{mt3-q17}
 \beta_2 - I_4= -\frac{BC(-1)g(B)g(A\overline{BC})g(\overline{A}C)g(A\overline{B})g(B\overline{C})\overline{A}B(x)}{q^2g(A)g(\overline{C})}.
\end{align}
Multiplying both numerator and denominator by $g(\overline{B}C)$ and then using Lemma \ref{g1} on $g(B\overline{C})g(\overline{B}C)$ and the fact that $B\neq C$, we have
\begin{align}\label{mt3-q23}
L&=\frac{1}{(q-1)}\sum_{\psi}\frac{g(A\psi)g(\overline{C\psi})g(\overline{B}C\psi)
g(\overline{\psi})}{g(A)g(\overline{C})g(\overline{B}C)}\psi(x) \\
&\times A\overline{B}(1-y){_{2}}F_1\left(\begin{array}{cc}
                                               A\overline{C}, & \overline{\psi}\\
                                               & B\overline{C\psi}
                                           \end{array}\mid y \right)^{\ast}+\beta_1 + \beta_2 -I_3-I_4 + \frac{q-1}{q}I_1 + \frac{q-1}{q}I_2 -\alpha_1.\notag
\end{align}
Employing \eqref{mt3-q19}, \eqref{mt3-q21}, \eqref{mt3-q17} and \eqref{mt3-q15} into \eqref{mt3-q23} we obtain
\begin{align}
L&=\frac{1}{(q-1)}\sum_{\psi}\frac{g(A\psi)g(\overline{C\psi})g(\overline{B}C\psi)
g(\overline{\psi})}{g(A)g(\overline{C})g(\overline{B}C)}\psi(x)\notag \\
&\hspace{1cm}\times A\overline{B}(1-y){_{2}}F_1\left(\begin{array}{cc}
                                               A\overline{C}, & \overline{\psi}\\
                                               & B\overline{C\psi}
                                           \end{array}\mid y \right)^{\ast} -\alpha_1.\notag
\end{align}
Using \eqref{mcCarty's_defn} and then rearranging the terms we have
\begin{align}
 L&=\frac{ A\overline{B}(1-y)}{(q-1)^2}\sum_{\psi,\chi}\frac{g(A\psi)g(\overline{C\psi})g(A\overline{C}\chi)g(\overline{B}C\psi\overline{\chi})g(\overline{\psi}\chi)
 g(\overline{\chi})}{g(A)g(\overline{B}C)g(A\overline{C})g(\overline{C})}\psi(x)\chi(y) - \alpha_1.\notag
\end{align}
If we apply the change of variables $\psi \mapsto \psi\chi$ and then rearranging the terms, we have
\begin{align}
 L&=\frac{ A\overline{B}(1-y)}{(q-1)^2}\sum_{\psi,\chi}\frac{g(A\psi\chi)g(\overline{B}C\psi)g(A\overline{C}\chi)g(\overline{C\psi\chi})g(\overline{\psi})
 g(\overline{\chi})}{g(A)g(\overline{B}C)g(A\overline{C})g(\overline{C})}\psi(x)\chi(xy) - \alpha_1.\notag
\end{align}
Finally, using  \eqref{def1} and  \eqref{mt3-q4} we obtain
\begin{align}
L&=A\overline{B}(1-y)F_{1}(A;\overline{B}C,A\overline{C};C;x,xy)^{*} - \frac{g(B)g(A\overline{B})\overline{B}(y)\overline{A}B(1-x)}{q\cdot g(A)}.\notag
\end{align}
Multiplying both side by $A(1-x)B(1-y)$, we complete the proof of the theorem.

\end{proof}

%%%%%%%%%%%%%%%%%%%%%%%%%%%%%%%%%%%%%%%%%%%%%%%%%%%%%%%%%%%%%%%%%%%%%%%%%%%%%%%%%%%%%%%%%%%%%%%

\end{document}